\documentclass[12pt,reqno]{amsart}
\usepackage{enumerate, latexsym, amsmath, amsfonts, amssymb, amsthm, color}
\def\pmod #1{\ ({\rm{mod}}\ #1)}

\def\bg{\bigg}
\def\({\bg(}
\def\){\bg)}

\def\gs{\geqslant}

\def\eq{\equiv}

\def\Ack{\medskip\noindent {\bf Acknowledgment}}
\theoremstyle{plain}
\newtheorem{theorem}{Theorem}

\newtheorem{lemma}{Lemma}

\newtheorem{conjecture}{Conjecture}
\theoremstyle{definition}

\theoremstyle{remark}
\newtheorem{remark}{Remark}

 \vspace{4mm}

\begin{document}
\title
[{Proof of some congruence conjectures of Z.-H. Sun}]
{Proof of some congruence conjectures of Z.-H. Sun involving Ap\'{e}ry-like numbers}

\author
[Guo-Shuai Mao] {Guo-Shuai Mao}

\address {(Guo-Shuai Mao) Department of Mathematics, Nanjing
University of Information Science and Technology, Nanjing 210044,  People's Republic of China\\
{\tt maogsmath@163.com  } }

\keywords{Congruences; Franel numbers; Bernoulli numbers; Euler numbers; Fermat quotient.
\newline \indent This research was supported by the National Natural Science Foundation of China (grant 12001288).}

 \begin{abstract} In this paper, we mainly prove the following conjecture of Z.-H. Sun \cite{SH20}:
 Let $p>3$ be a prime. Then
 $$
 \sum_{k=0}^{p-1}\binom{2k}k\frac{3k+1}{(-16)^k}f_k\equiv(-1)^{(p-1)/2}p+p^3E_{p-3}\pmod{p^4},
 $$
 where $f_n=\sum_{k=0}^n\binom{n}k^3$ and $E_n$ stand for the $n$th Franel number and $n$th Euler number respectively.
\end{abstract}

\maketitle

\section{Introduction}
\setcounter{lemma}{0}
\setcounter{theorem}{0}
\setcounter{corollary}{0}
\setcounter{remark}{0}
\setcounter{equation}{0}
\setcounter{conjecture}{0}
In 1894, Franel \cite{F} found that the numbers
$$
f_n=\sum_{k=0}^n\binom{n}k^3\ \ (n=0,1,2,\ldots)
$$
satisfy the recurrence relation (cf. \cite[A000172]{SL}):
$$
(n+1)^2f_{n+1}=(7n^2+7n+2)f_n+8n^2f_{n-1}\ \ (n=1,2,3,\ldots).
$$
These numbers are now called Franel numbers. Callan \cite{C} found a combinatorial interpretation of the Franel numbers. The Franel numbers play important roles in combinatorics and number theory. The sequence $(f_n)_{n\gs0}$ is one of the five sporadic sequences (cf. \cite[Section 4]{Z})
which are integral solutions of certain Ap\'ery-like recurrence equations and closely related to the theory of modular forms.
In 2013, Sun \cite{S13} revealed some unexpected connections between numbers $f_n$ and representations of primes $p\eq1\pmod3$ in the form $x^2+3y^2$ with $x,y\in\mathbb{Z}$, for example, Z.-W. Sun \cite[(1.2)]{S13} showed that
$$
\sum_{k=0}^{p-1}\frac{f_k}{2^k}\equiv\sum_{k=0}^{p-1}\frac{f_k}{(-4)^k}\equiv 2x-\frac{p}{2x}\pmod{p^2}.
$$
For more studies on Franel numbers, we refer the readers to \cite{G,G1,Liu,M,S13,sun}.

For $n\in\mathbb{N}$, define
$$H_n:=\sum_{0<k\leq n}\frac1k,\ \ \mbox{and}\ \ H_n^{(2)}:=\sum_{0<k\leq n}\frac1{k^2}\ \ \mbox{with}\ \  \ H_0=H_0^{(2)}=0,$$
where $H_n$ with $n\in\mathbb{N}$ are often called the classical harmonic numbers. Let $p>3$ be a prime. J. Wolstenholme \cite{wolstenholme-qjpam-1862} proved that
\begin{align}\label{hp-1}
H_{p-1}\equiv0\pmod{p^2}\ \mbox{and}\ H_{p-1}^{(2)}\equiv0\pmod p,
\end{align}
which imply that
\begin{align}\label{2p1p1}
\binom{2p-1}{p-1}\equiv1\pmod{p^3}.
\end{align}
In 1990 Glaisher \cite{gl1,gl2} showed further that
\begin{equation}\label{gl}
\binom{2p-1}{p-1}\equiv1-\frac23p^3B_{p-3}\pmod{p^4},
\end{equation}
where $B_0, B_1, B_2,\ldots$ are Bernoulli numbers. (See \cite{IR} for an introduction to Bernoulli numbers).

\noindent Recall that the Euler numbers $\{E_n\}$ are defined by
$$\frac{2e^t}{e^{2t}+1}=\sum_{n=0}^\infty E_n\frac{t^n}{n!}\ (|t|<\frac{\pi}2).$$
In \cite{SH20}, Z.-H. Sun studied some Ap\'{e}ry-like numbers, he got many beautiful congruences involving these numbers and proposed a lot of conjectures at the end of the paper, one of the conjectures is,
\begin{conjecture}\label{Conj1.1} {\rm (\cite[Conjecture 4.29]{SH20})} Let $p>3$ be a prime. Then
$$
 \sum_{k=0}^{p-1}\binom{2k}k\frac{3k+1}{(-16)^k}f_k\equiv(-1)^{(p-1)/2}p+p^3E_{p-3}\pmod{p^4}.
 $$
\end{conjecture}
\begin{remark} In \cite[Conjecture 4.23]{sun13}, Sun conjectured the congruence in Conjecture \ref{Conj1.1} modulo $p^2$ and Guo \cite{G1} proved it modulo $p^3$.
\end{remark}
In this paper, our first goal is to prove the above conjecture.
 \begin{theorem}\label{Th1.1} Conjecture \ref{Conj1.1} is true.
\end{theorem}
Recall that another Ap\'{e}ry-like numbers
$$
T_n=\sum_{k=0}^n\binom{n}k^2\binom{2k}n^2\ \ (n=0,1,2,\ldots)
$$
which satisfy the recurrence relation:
$$
(n+1)^3T_{n+1}=(2n+1)(12n(n+1)+4)T_n-16n^3T_{n-1}\ \ (n\geq1).
$$
Z.-H. Sun \cite[Theorem 2.1]{sjdea} obtain that for any nonnegative integer $n$,
\begin{align}\label{sid}
T_n=\sum_{k=0}^{\lfloor n/2\rfloor}\binom{2k}k^2\binom{4k}{2k}\binom{n+2k}{4k}4^{n-2k}.
\end{align}
In the same paper he showed that for any prime $p>3$, we have
$$
\sum_{k=0}^{p-1}(2k+1)\frac{T_k}{4^k}\equiv p\pmod{p^4}
$$
and
$$
\sum_{k=0}^{p-1}(2k+1)\frac{T_k}{(-4)^k}\equiv(-1)^{(p-1)/2}p\pmod{p^3}.
$$
Our second goal is to prove the following two congruences which were conjectured by Z.-H. Sun \cite[Conjecture 2.4]{sjdea}:
\begin{theorem}\label{Th1.2} Let $p>3$ be a prime. Then
\begin{equation}\label{t4}
\sum_{k=0}^{p-1}(2k+1)\frac{T_k}{4^k}\equiv p+\frac76p^4B_{p-3}\pmod{p^5}
\end{equation}
and
\begin{equation}\label{t-4}
\sum_{k=0}^{p-1}(2k+1)\frac{T_k}{(-4)^k}\equiv(-1)^{(p-1)/2}p+p^3E_{p-3}\pmod{p^4}.
\end{equation}
\end{theorem}
\begin{remark}
A few days after I submitted this paper to arXiv, Z.-H. Sun told me that the two congruences in Theorem \ref{Th1.2} have been proved by Liu \cite{Lt}, and then I found that our proof differs from his, because we could obtain the results directly by using one congruence involving harmonic numbers which was obtained by the author, C. Wang and J. Wang \cite{mww} and by the results in \cite{sjdea} and \cite[(1.7)]{sun13}. And one of the referees suggested me to delete this theorem and related paragraphs and references, while another referee thought that I keep this theorem in this paper is a reasonable point, so I think I may keep this theorem and its proof in this paper.
\end{remark}
We are going to prove Theorem \ref{Th1.1} in Section 2. And the last Section is devoted to prove Theorem \ref{Th1.2}. Our proofs make use of some new combinatorial identities which were found by the package \texttt{Sigma} \cite{S} via the software \texttt{Mathematica}.
 \section{Proof of Theorem \ref{Th1.1}}
 \setcounter{lemma}{0}
\setcounter{theorem}{0}
\setcounter{corollary}{0}
\setcounter{remark}{0}
\setcounter{equation}{0}
\setcounter{conjecture}{0}
We need the following identity due to Z.-W. Sun \cite{S13}:
\begin{align}\label{2.1}f_n=\sum_{k=0}^n\binom{n+2k}{3k}\binom{2k}k\binom{3k}k(-4)^{n-k}.
\end{align}
We also rely on the following identity which can be proved by induction on $n$:
\begin{equation}\label{3k1id}
\sum_{k=j}^{n-1}\frac{3k+1}{4^k}\binom{2k}k\binom{k+2j}{3j}=\frac{2n(3j+1)\binom{n+2j}{3j+1}\binom{2n}n}{(2j+1)4^n}.
\end{equation}
By loading the package \texttt{Sigma} in the software \texttt{Mathematica}, we find an interesting and useful identity:
\begin{lemma}\label{id} For any positive integer $n$, we have
\begin{align*}
&\sum_{j=0}^{n-1}\binom{n}j(-1)^j\frac{H_{2j}-H_j}{2j+1}\\
&=\frac{4^n}{2(2n+1)\binom{2n}n}\sum_{k=1}^n\frac{\binom{2k}k}{k4^k}+\frac{(-1)^nH_n}{2n+1}-\frac{4^nH_n}{2(2n+1)\binom{2n}n}-\frac{(-1)^nH_{2n}}{2n+1}.
\end{align*}
\end{lemma}
\begin{proof} We can prove the above identitiy by using the package \verb"Sigma" \cite{S} in Mathematica as follows, other similar identities in the whole paper also can be proved in the same way. We first insert
\begin{align*}
\text{In}[1]:&=\text{mysum}=\text{SigmaSum}[\text{SigmaBinomial}[n,k]*(-1)^k\\&*(\text{SigmaHNumber}[2*k]-\text{SigmHNumber}[k])\\&/(2*k+1),\{k,0,n-1\}]\\
\text{Out}[1]=&\sum_{k=0}^{n-1}\frac{\binom{n}k(-1)^k(H_{2k}-H_k)}{2k+1}
\end{align*}
Then we can obtain the recurrence for the above sum:
\begin{align*}
\text{In}[2]:=&\text{rec=GenerateRecurrence[mysum]}\\
\text{Out}[2]=&\{-4(1+n)^2\text{SUM}[n]+2(3+2n)^2\text{SUM}[1+n]\\&
-(3+2n)(5+2n)\text{SUM}[2+n]==\frac1{2(2+n)}+\\&
\frac{(19+24n+8n^2)(-1)^n}{2(1+n)(2+n)(1+2n)}+\frac{(-13-32n-16n^2)(-1)^nH_n}{1+2n}\\&+\frac{(13+32n+16n^2)(-1)^nH_{2n}}{1+2n}\}
\end{align*}
Now we solve the above recurrence:
\begin{flalign*}
&\text{In}[3]:=\text{resol=SolveRecurrence[rec[[1]], SUM[n]]}&&\\
&\text{Out}[3]=\bigg\{\left\{0,\frac{\prod_{l=1}^n\frac{2l}{-1+2l}}{1+2n}\right\},\bigg\{0,\frac{H_n\prod_{l=1}^n\frac{2l}{-1+2l}}{-1-2n}\bigg\},\\
&\bigg\{1,\frac{\bigg(\prod_{l=1}^n\frac{2l}{-1+2l}\bigg)\sum_{l=1}^n\frac{\prod_{l=1}^n\frac{-1+2l}{2l}}{l}}{2(1+2n)}+(-1)^n\left(\frac{H_n}{1+2n}-\frac{H_{2n}}{1+2n}\right)\bigg\}\bigg\}&&
\end{flalign*}
At last, we obtain another form of mysum by finding the linear combination of the solutions:
\begin{flalign*}
\text{In}[4]:=&\text{FindLinearCombination[resol, mysum, 2]}&&\\
\text{Out}[4]=&\frac{\bigg(\prod_{l=1}^n\frac{2l}{-1+2l}\bigg)\sum_{l=1}^n\frac{\prod_{l=1}^n\frac{-1+2l}{2l}}{l}}{2(1+2n)}+\frac{(-1)^nH_n}{1+2n}\\&-\frac{H_n\prod_{l=1}^n\frac{2l}{-1+2l}}{2(1+2n)}-\frac{(-1)^nH_{2n}}{1+2n}
\end{flalign*}
Thus we get the desired identity by simple computation.

\noindent Therefore, the proof of Lemma \ref{id} is finished.
\end{proof}
\begin{lemma}\label{p2j} Let $p>2$ be a prime. If $0\leq j\leq (p-3)/2$, then we have
$$
(3j+1)\binom{3j}j\binom{p+2j}{3j+1}\equiv p(-1)^j(1+pH_{2j}-pH_j)\pmod{p^3}.
$$
If $(p+1)/2\leq j\leq p-1$, then
$$
(3j+1)\binom{2j}j\binom{3j}j\binom{p+2j}{3j+1}\equiv 2p(-1)^j\binom{2j}j\pmod{p^3}.
$$
\end{lemma}
\begin{proof}If $0\leq j\leq (p-3)/2$, then we have
\begin{align*}
(3j+1)\binom{3j}j\binom{p+2j}{3j+1}&=\frac{(p+2j)\cdots(p+1)p(p-1)\cdots(p-j)}{j!(2j)!}\\
&\equiv\frac{p(2j)!(1+pH_{2j})(-1)^{j}(j)!(1-pH_j)}{j!(2j)!}\\
&\equiv p(-1)^j(1+pH_{2j}-pH_j)\pmod{p^3}.
\end{align*}
If $(p+1)/2\leq j\leq p-1$, then we have $\binom{2j}j\equiv0\pmod p$, and hence
\begin{align*}
&(3j+1)\binom{2j}j\binom{3j}j\binom{p+2j}{3j+1}\\
=&\binom{2j}j\frac{(p+2j)\cdots(2p+1)(2p)(2p-1)\cdots(p+1)p(p-1)\cdots(p-j)}{j!(2j)!}\\
\equiv&\binom{2j}j\frac{2p^2(2j)\cdots(p+1)(p-1)!(-1)^{j}(j)!}{j!(2j)!}=2p(-1)^j\binom{2j}j\pmod{p^3}.
\end{align*}
Now the proof of Lemma \ref{p2j} is complete.
\end{proof}
In view of (\ref{2.1}) and (\ref{3k1id}), we have
\begin{align*}
&\sum_{k=0}^{p-1}\frac{3k+1}{(-16)^k}\binom{2k}kf_k\\
&=\sum_{k=0}^{p-1}\frac{3k+1}{(-16)^k}\binom{2k}k\sum_{j=0}^k\binom{k+2j}{3j}\binom{3j}j\binom{2j}j(-4)^{k-j}\\
&=\sum_{j=0}^{p-1}\frac{\binom{2j}j\binom{3j}j}{(-4)^j}\sum_{k=j}^{p-1}\frac{3k+1}{4^k}\binom{2k}k\binom{k+2j}{3j}\\
&=\sum_{j=0}^{p-1}\frac{\binom{2j}j\binom{3j}j}{(-4)^j}\frac{2p(3j+1)\binom{p+2j}{3j+1}\binom{2p}p}{(2j+1)4^j}\\
&=\frac{p}{4^{p-1}}\binom{2p-1}{p-1}\sum_{j=0}^{p-1}\frac{\binom{2j}j(3j+1)\binom{3j}j\binom{p+2j}{3j+1}}{(2j+1)(-4)^j}.
\end{align*}
This, with Lemma \ref{p2j} yields that
\begin{equation}\label{main}
\sum_{k=0}^{p-1}\frac{3k+1}{(-16)^k}\binom{2k}kf_k\equiv\frac{p}{4^{p-1}}(S_1+S_2+S_3)\pmod{p^4},
\end{equation}
where
$$
S_1=p\sum_{j=0}^{(p-3)/2}\frac{\binom{2j}j}{4^j}\frac{1+pH_{2j}-pH_j}{2j+1},
$$
$$
S_2=\frac{\binom{2p-1}{p-1}\binom{p-1}{(p-1)/2}^2}{(-4)^{p-1)/2}}\ \ \mbox{and}\ \ S_3=2p\sum_{j=(p+1)/2}^{p-1}\frac{\binom{2j}j}{(2j+1)4^j}.
$$
It is well known the Morley's congruence \cite{Mor}:
\begin{equation}\label{mor}
\binom{p-1}{(p-1)/2}\equiv(-1)^{(p-1)/2}4^{p-1}\pmod{p^3}\ \ \mbox{for}\ \ p>3.
\end{equation}
And Calitz \cite{Ca} further showed that
\begin{equation}\label{ca}
\binom{p-1}{(p-1)/2}\equiv(-1)^{(p-1)/2}\left(4^{p-1}+\frac{p^3}{12}B_{p-3}\right)\pmod{p^4}.
\end{equation}
It is easy to see that
\begin{align}\label{main1}
S_1&=p\sum_{j=0}^{\frac{p-3}2}\frac{\binom{2j}j}{(2j+1)4^j}+p^2\sum_{j=0}^{\frac{p-3}2}\frac{\binom{2j}j}{4^j}\frac{H_{2j}-H_j}{2j+1}\notag\\
&\equiv p\sum_{j=0}^{\frac{p-3}2}\frac{\binom{2j}j}{(2j+1)4^j}+p^2\sum_{j=0}^{\frac{p-3}2}\binom{n}j(-1)^j\frac{H_{2j}-H_j}{2j+1}\pmod{p^3}.
\end{align}
We know that Tauraso \cite{T1} and Sun \cite[(1.5)]{s11} respectively proved that
$$
\sum_{k=1}^{p-1}\frac{\binom{2k}k}{k4^k}\equiv-H_{(p-1)/2}\pmod{p^3},\ \sum_{k=\frac{p+1}2}^{p-1}\frac{\binom{2k}k}{k4^k}\equiv(-1)^{\frac{p-1}2}2pE_{p-3}\pmod{p^2}.
$$
And it is easy to check that by (\ref{mor}), we have
$$
(-1)^{\frac{p-1}2}\binom{p-1}{\frac{p-1}2}-2^{p-1}\equiv 4^{p-1}-2^{p-1}=2^{p-1}pq_p(2)\equiv pq_p(2)\pmod{p^2},
$$
where $q_p(a)=(a^{p-1}-1)/p$ stands for the Fermat quotient.

\noindent These, with Lemma \ref{id}, (\ref{hp-1}), (\ref{mor}) and $H_{(p-1)/2}\equiv-2q_p(2)\pmod p$ \cite{lehmer} yield that
\begin{align*}
&\sum_{j=0}^{\frac{p-3}2}\binom{n}j(-1)^j\frac{H_{2j}-H_j}{2j+1}\\
&=\frac{2^{p-1}\sum_{k=1}^{\frac{p-1}2}\frac{\binom{2k}k}{k4^k}}{2p\binom{p-1}{\frac{p-1}2}}+\frac{(-1)^{\frac{p-1}2}H_{\frac{p-1}2}}{p}-\frac{2^{p-1}H_{\frac{p-1}2}}{2p\binom{p-1}{\frac{p-1}2}}-\frac{(-1)^{\frac{p-1}2}H_{p-1}}{p}\\
&\equiv\frac{((-1)^{\frac{p-1}2}\binom{p-1}{\frac{p-1}2}-2^{p-1})H_{\frac{p-1}2}}{p\binom{p-1}{\frac{p-1}2}}-E_{p-3}\\
&\equiv(-1)^{(p-1)/2}q_p(2)H_{(p-1)/2}-E_{p-3}\equiv(-1)^{\frac{p+1}2}2q^2_p(2)-E_{p-3}\pmod p.
\end{align*}
Substituting this into (\ref{main1}) and by \cite[(1.1)]{s11jnt}, we have
\begin{equation*}
S_1\equiv(-1)^{\frac{p+1}2}pq_p(2)+(-1)^{\frac{p+1}2}2p^2q^2_p(2)-p^2E_{p-3}\pmod{p^3}.
\end{equation*}
This, with $2^{p-1}=1+pq_p(2)$ and Fermat's little theorem yields that
\begin{align}\label{s1}
\frac{p}{4^{p-1}}S_1&\equiv\frac{p}{4^{p-1}}((-1)^{\frac{p+1}2}pq_p(2)+(-1)^{\frac{p+1}2}2p^2q^2_p(2)-p^2E_{p-3})\notag\\
&\equiv(-1)^{\frac{p+1}2}p^2q_p(2)-p^3E_{p-3}\pmod{p^4}.
\end{align}
It is easy to check that by (\ref{2p1p1}), (\ref{mor}) and $2^{p-1}=1+pq_p(2)$, we have
\begin{equation}\label{s2}
\frac{p}{4^{p-1}}S_2\equiv(-1)^{\frac{p-1}2}p2^{p-1}=(-1)^{\frac{p-1}2}p+(-1)^{\frac{p-1}2}p^2q_p(2)\pmod{p^4}.
\end{equation}
In view of \cite[(1.2)]{s11jnt}, we have
$$
\frac{p}{4^{p-1}}S_3\equiv2p^3E_{p-3}\pmod{p^4}.
$$
Therefore, combining this with (\ref{main}), (\ref{s1}) and (\ref{s2}), we immediately obtain the desired result
$$
\sum_{k=0}^{p-1}\binom{2k}k\frac{3k+1}{(-16)^k}f_k\equiv(-1)^{(p-1)/2}p+p^3E_{p-3}\pmod{p^4}.
$$
\section{Proof of Theorem \ref{Th1.2}}
\setcounter{lemma}{0}
\setcounter{theorem}{0}
\setcounter{corollary}{0}
\setcounter{remark}{0}
\setcounter{equation}{0}
\setcounter{conjecture}{0}
In \cite[(1.7)]{sp}, Sun showed that
\begin{equation}\label{2k+1}
\sum_{k=0}^{\frac{p-3}2}\frac{\binom{2k}k^2}{(2k+1)16^k}\equiv-2q_p(2)-pq^2_p(2)+\frac5{12}p^2B_{p-3}\pmod{p^3}
\end{equation}
and in \cite[(1.7)]{sun13}, he proved that
\begin{equation}\label{p3}
\sum_{k=0}^{(p-1)/2}\frac{\binom{2k}k^2}{16^k}\equiv(-1)^{(p-1)/2}+p^2E_{p-3}\pmod{p^3}.
\end{equation}
In view of the following congruence which was confirmed by the author, C. Wang and J. Wang \cite{mww}
$$
\sum_{k=1}^{(p-1)/2}\frac{\binom{2k}k^2}{k16^k}H_{2k}^{(2)}\equiv-\frac52B_{p-3}\pmod p,
$$
we can immediately obtain the following important lemma.
\begin{lemma}\label{h2k} For any prime $p>3$, we have
\begin{equation*}
\sum_{k=0}^{(p-3)/2}\frac{\binom{2k}k^2}{(2k+1)16^k}H_{2k}^{(2)}\equiv-\frac54B_{p-3}\pmod p.
\end{equation*}
\end{lemma}
\begin{proof} In view of (\ref{hp-1}) and $\binom{2k}k/(-4)^k\equiv\binom{(p-1)/2}k\pmod p$ for each $0\leq k\leq(p-1)/2$, we have
\begin{align*}
\sum_{k=0}^{\frac{p-3}2}\frac{\binom{2k}k^2}{(2k+1)16^k}H_{2k}^{(2)}&\equiv\sum_{k=0}^{\frac{p-3}2}\frac{\binom{\frac{p-1}2}k^2}{2k+1}H_{2k}^{(2)}=\sum_{j=1}^{\frac{p-1}2}\frac{\binom{(p-1)/2}k^2H_{p-1-2k}^{(2)}}{p-2k}\\
&\equiv-\frac12\sum_{j=1}^{\frac{p-1}2}\frac{\binom{2k}k^2(H_{p-1}^{(2)}-H_{2k}^{(2)})}{k16^k}\equiv\frac12\sum_{j=1}^{\frac{p-1}2}\frac{\binom{2k}k^2H_{2k}^{(2)}}{k16^k}\\
&\equiv-\frac54B_{p-3}\pmod p,
\end{align*}
where we used the following congruence
\begin{equation}\label{p12k}
H_{p-1-2k}^{(2)}=\sum_{j=1}^{p-1-2k}\frac1{j^2}=\sum_{j={2k+1}}^{p-1}\frac1{(p-j)^2}\equiv H_{p-1}^{(2)}-H_{2k}^{(2)}\pmod p.
\end{equation}
So the proof of Lemma \ref{h2k} is finished.
\end{proof}
\noindent {\it Proof of (\ref{t4})}.\ In view of \cite[Theorem 2.6]{sjdea}, Sun already proved that
\begin{align*}
\sum_{k=0}^{p-1}(2k+1)\frac{T_k}{4^k}\equiv p^2\sum_{k=0}^{(p-1)/2}\frac{\binom{2k}k^2}{(2k+1)16^k}(1-p^2H_{2k}^{(2)})\pmod{p^5}.
\end{align*}
This, with (\ref{2k+1}), Lemma \ref{h2k}, (\ref{ca}), \cite[Corollary 5.1]{s2000} and Fermat's little theorem yields that
\begin{align*}
&\sum_{k=0}^{p-1}(2k+1)\frac{T_k}{4^k}\\
&\equiv p^2\sum_{k=0}^{\frac{p-3}2}\frac{\binom{2k}k^2}{(2k+1)16^k}-p^4\sum_{k=0}^{\frac{p-3}2}\frac{\binom{2k}k^2}{(2k+1)16^k}H_{2k}^{(2)}+\frac{p\binom{p-1}{\frac{p-1}2}^2(1-p^2H_{p-1}^{(2)})}{4^{p-1}}\\
&\equiv-2p^2q_p(2)-p^3q^2_p(2)+\frac5{12}p^4B_{p-3}+\frac54p^4B_{p-3}+p4^{p-1}-\frac{p^4}2B_{p-3}\\
&\equiv p+\frac76p^4B_{p-3}\pmod{p^5},
\end{align*}
where we also used a fact that $4^{p-1}=1+2pq_p(2)+p^2q^2_p(2)$. So we finish the proof of (\ref{t4}).

\noindent {\it proof of (\ref{t-4})}.\ \ Also, Sun \cite[Theorem 2.7]{sjdea} already showed that
\begin{equation}\label{p5}
\sum_{k=0}^{p-1}(2k+1)\frac{T_k}{(-4)^k}\equiv p\sum_{k=0}^{(p-1)/2}\frac{\binom{2k}k^2}{16^k}(1-p^2H_{2k}^{(2)})\pmod{p^5}.
\end{equation}
As seen in (\ref{p12k}) and by using (\ref{hp-1}), we have $H_{p-1-2k}^{(2)}\equiv-H_{2k}^{(2)}\pmod p$, and hence
\begin{align*}
\sum_{k=0}^{(p-1)/2}\frac{\binom{2k}k^2}{16^k}H_{2k}^{(2)}&\equiv\sum_{k=0}^{(p-1)/2}\binom{\frac{p-1}2}k^2H_{2k}^{(2)}=\sum_{j=0}^{(p-1)/2}\binom{\frac{p-1}2}j^2H_{p-1-2j}^{(2)}\\
&\equiv-\sum_{k=0}^{(p-1)/2}\binom{\frac{p-1}2}k^2H_{2k}^{(2)}\pmod p.
\end{align*}
Thus,
$$
\sum_{k=0}^{(p-1)/2}\frac{\binom{2k}k^2}{16^k}H_{2k}^{(2)}\equiv0\pmod p.
$$
This, with (\ref{p3}) and (\ref{p5}) yields the desired result
$$
\sum_{k=0}^{p-1}(2k+1)\frac{T_k}{(-4)^k}\equiv p(-1)^{(p-1)/2}+p^3E_{p-3}\pmod{p^4}.
$$
Now the proof of Theorem \ref{Th1.1} is complete.\qed

\Ack. The author would like to thank the anonymous referees for helpful comments. The author was supported by the National Natural Science Foundation of China (12001288).


\begin{thebibliography}{ST10}
    \bibitem{Ca} L.Calitz, {\it A theorem of Glaisher}, Canadian J. Math. {\bf 5} (1953), 306--316.
    \bibitem{C} D. Callan, {\it A combinatorial interpretation for an identity of Barrucand}, J.Integer Seq. {\bf 11} (2008), Article 08.3.4, 3pp (electronic).
    \bibitem{F} J. Franel, {\it On a question of Laisant}, L'Interm\'{e}diaire des Math\'{e}maticiens {\bf 1} (1894), 45--47.
    \bibitem{gl1} J.W.L. Glaisher, {\it Congruences relating to the sums of products of the first $n$ numbers and to other sums of products}, Quart. J. Math. {\bf 31} (1900), 1--35.
    \bibitem{gl2} J.W.L. Glaisher, {\it On the residues of the sums of products of the first $p-1$ numbers, and their powers, to modulus $p^2$ or $p^3$}, Quart. J. Math. {\bf 31} (1900), 321--353.
    \bibitem{G} V.J.W. Guo, {\it Proof of two conjectures of Sun on congruences for Franel numbers}, Integral Transforms Spec. Funct. {\bf 24} (2013), 532--539.
    \bibitem{G1} V.J.W. Guo, {\it Proof of a superconjecture conjectured by Z.-H. Sun}, Integral Transforms Spec. Funct. {\bf 25} (2014), 1009--1015.
    \bibitem{IR} K. Ireland and M. Rosen, A classical Introduction to Modern Number Theory, second ed., Graduate Texts in Math., Vol. 84, Springer, New York, 1990.
    \bibitem{lehmer} E. Lehmer, {\it On congruences involving Bernoulli numbers and the quotients of Fermat and Wilson}, Ann. Math.  {\bf 39}(1938), 350--360.
    \bibitem{Liu} J.-C. Liu, {\it On two congruences involving Franel numbers}, Rev. R. Acad. Cienc. Exactas F\'is. Nat., Ser. A Math. {\bf 114} (2020), Art.201.
    \bibitem{Lt} J.-C. Liu, {\it On two supercongruences for sums of Ap\'{e}ry-like numbers}, Rev. R. Acad. Cienc. Exactas F\'{\i}s. Nat., Ser. A Mat. {\bf 115} (2021), Art. 151.
    \bibitem{M} G.-S. Mao, {\it On two congruences involving Ap\'{e}ry and Franel numbers}, Results Math. {\bf 75} (2020), Art 159.
    \bibitem{mww} G.-S. Mao, C. Wang and J. Wang, {\it Symbolic summation methods and congruences invlving harmonic numbers}, C. R. Acad. Sci. Paris, Ser. I, {\bf 357} (2019), 756--765.
    \bibitem{Mor} F. Morley, {\it Note on the congruence $2^{4n}\equiv(-1)^n(2n)!/(n!)^2$, where $2n+1$ is a prime}, Ann. Math. {\bf 9} (1895), 168--170.
    \bibitem{S}C. Schneider, {\it Symbolic summation assists combinatorics}, S\'em. Lothar. Combin. {\bf 56} (2007), Article B56b.
    \bibitem{SL}N.J.A. Sloane, Sequence A000172 in OEIS (On-Line Encyclopedia of Integer Sequences), http://oeis.org/A000172.
    \bibitem{s2000} Z.-H. Sun, {\it Congruences concerning Bernoulli numbers and Bernoulli polynomials}, Discrete Appl. Math. {\bf 105} (2000), no.1-3, 193--223.
    \bibitem{sun13} Z.-H. Sun, {\it Some conjectures on congruences}, preprint, arXiv:1103.5384v5.
    \bibitem{sjdea} Z.-H. Sun, {\it Super congruences for two Apery-like sequences}, J. Difference. Equ. Appl. {\bf 24} (2018), no.10, 1685--1713.
    \bibitem{SH20} Z.-H. Sun, {\it Congruences involving binomial coefficients and Apery-like numbers}, Publ. Math. Debrecen {\bf 96} (2020), no.3--4, 315--346.
    \bibitem{s11jnt}Z.-W. Sun, {\it On congruences related to central binomial coefficients}, J. Number Theory {\bf 13} (2011), no.11, 2219--2238.
    \bibitem{s11}Z.-W. Sun, {\it Super congruences and Euler numbers}, Sci. China Math. {\bf 54} (2011), no.12, 2509--2535.
    \bibitem{S13}Z.-W. Sun, {\it Connections between $p=x^2+3y^2$ and Franel numbers}, J. Number Theory {\bf 133} (2013), 2914--2928.
    \bibitem{sun}Z.-W. Sun, {\it Congruences for Franel numbers}, Adv. in Appl. Math. {\bf 51} (2013), no.4, 524--535.
    \bibitem{sp}Z.-W. Sun, {\it $p$-adic congruences motivated by series}, J. Number Theory {\bf 134} (2014), no.1, 181--196.
    \bibitem{T1}R. Tauraso, {\it Congruences involving alternating multiple harmonic sum}, Elrctron. J. Combin. {\bf 17} (2010), \#R16, 11pp (electronic).
    \bibitem{wolstenholme-qjpam-1862}J. Wolstenholme, {\it On certain properties of prime numbers}, Quart. J. Pure Appl. Math. {\bf 5} (1862), 35--39.
    \bibitem{Z} D. Zagier, {\it Integral solutions of Ap\'ery-like recurrence equations},
    in: Groups and Symmetries: from Neolithic Scots to John McKay, CRM Proc. Lecture Notes 47, Amer. Math. Soc., Providence, RI, 2009, pp. 349--366.
     \end{thebibliography}
     \end{document}